\documentclass[11pt]{amsart}
\usepackage{amsmath, amsthm, amssymb, amscd, graphicx, hyperref, amsfonts, epstopdf}
\newtheorem{thm}{Theorem}[section]
\newtheorem{cor}[thm]{Corollary}
\newtheorem{lem}[thm]{Lemma}
\newtheorem{clm}[thm]{Claim}

\theoremstyle{definition}

\theoremstyle{remark}
\newtheorem{rem}[thm]{Remark}

\numberwithin{equation}{section}

\def\Q{{\mathbb Q}}

\newcommand{\ra}{\rightarrow}

\def\MCG{{\mathcal MCG}}

\begin{document}
\title{$\mathcal{C}_{sep}(S_{2,0})$ is $\delta$-hyperbolic}
\author{Harold Sultan}
\address{Department  of Mathematics\\Columbia University\\
New York\\NY 10027}
\email{HSultan@math.columbia.edu}
\date{\today}
\begin{abstract}  A proof that $\mathcal{C}_{sep}(S_{2,0})$ has a quasi-distance formula and is $\delta$-hyperbolic using tools of Masur-Schleimer, \cite{masurschleimer}. This provides a proof in the affirmative of Conjecture 2.48 of \cite{schleimer}.  
 \end{abstract} 
\maketitle

 \tableofcontents
 
\section{Introduction}
It is well known that the curve complex $\mathcal{C}(S)$ is $\delta$-hyperbolic for all surfaces of positive complexity, see \cite{MM1}.  On the other hand, the separating curve complex $\mathcal{C}_{sep}(S)$ in general is not $\delta$-hyperbolic.  In particular, for all closed surfaces $S=S_{g,0}$ with genus $g \geq 3,$ as noted in \cite{schleimer}, $\mathcal{C}_{sep}(S)$ contains natural non-trivial \emph{quasi-flats}, or quasi-isometric embeddings of Euclidean flats; an obstruction to hyperbolicity.  For $S_{2,0}$ however, unlike closed surfaces of higher genus, there are no natural non-trivial quasi-flats.  Given this context, Schleimer conjectures that $\mathcal{C}_{sep}(S_{2,0})$ is $\delta$-hyperbolic; see \cite{schleimer} Conjecture 2.48.  In this paper, we prove this conjecture in the affirmative.  Note that the natural embedding $i:\mathcal{C}_{sep}(S) \ra \mathcal{C}(S)$ is known not to be a quasi-isometric embedding for all surfaces, and hence the proof of the conjecture does not follow from the hyperbolicity of the curve complex, \cite{MM1}.  

\begin{rem} While a proof that $\mathcal{C}_{sep}(S_{2,0})$ is $\delta$-hyperbolic is implicit in the work of Brock-Masur, \cite{brockmasur}, it is somewhat hidden, and so in this paper we present an alternative proof of this fact which is independent of their results.  In fact, since putting this paper on the arXiv, we have been informed that a very recent paper of Ma, \cite{ma}, proves the $\delta$-hyperbolicity of $\mathcal{C}_{sep}(S_{2,0})$ using the aforementioned work of Brock-Masur.
\end{rem}

The ideas in this paper are similar to, as well as motivated by, work of Masur-Schleimer in \cite{masurschleimer}.  Specifically, in \cite{masurschleimer} axioms are established for when a combinatorial complex has a quasi-distance formula and is $\delta$-hyperbolic.  In particular, Masur and Schleimer use these axioms to prove that the disk complex and the arc complex are $\delta$-hyperbolic.  While due to a technicality, the Masur-Schleimer axioms do not all hold in the case of $\mathcal{C}_{sep}(S_{2,0}),$ nonetheless, with enough care in this paper we are able to show by a direct argument that $\mathcal{C}_{sep}(S_{2,0})$ has a quasi-distance formula.  Furthermore, adapting the proof of $\delta$-hyperbolicity of Masur-Schleimer slightly, in conjunction with ideas used in proving the quasi-distance formula for $\mathcal{C}_{sep}(S_{2,0})$ we conclude by showing that $\mathcal{C}_{sep}(S_{2,0})$ is $\delta$-hyperbolic.  

The outline of the paper is as follows.  In the second section relevant background material is introduced.  The third section contains the core content of the paper including a proof of the quasi-distance formula for $\mathcal{C}_{sep}(S_{2,0})$ as well as a proof of $\delta$-hyperbolicity.  In the final section, a potential alternative proof of $\delta$-hyperbolicity is suggested.
 
  \subsection*{Acknowledgements}
$\\$
$\indent$
I want to express my gratitude to my advisors Jason Behrstock and Walter Neumann for their extremely helpful advice and insights throughout my research.  I would also like to acknowledge Saul Schleimer for useful conversations relevant to this paper.  All figures were drawn in Adobe Illustrator.  
\section{Preliminaries}
\subsection{Curve Complex and Separating Curve Complex} Given any surface of finite type, $S=S_{g,n},$ that is a genus $g$ surface with $n$ boundary components (or punctures), the \emph{complexity} of $S,$ denoted $\xi(S),$ is a topological invariant defined to be $3g-3+n.$  A simple closed curve in $S$ is \emph{peripheral} if it bounds a disk containing at most one boundary component; a non-peripheral curve is \emph{essential}.  Throughout the paper, we will use the term curve to refer to a geodesic representative of an isotopy class of a simple closed curve on a hyperbolic surface of finite type. 

For $S$ any surface with positive complexity, the \emph{curve complex} of $S,$ denoted $\mathcal{C}(S),$ is the simplicial complex obtained by associating to each curve a 0-cell, and more generally a k-cell to each unordered tuple of $k+1$ disjoint curves, or \emph{multicurves}.  In the special case of low complexity surfaces which do not admit disjoint curves, we relax the notion of adjacency to allow edges between vertices corresponding to curves which intersect minimally on the given surface.  Along similar lines, given a curve $\gamma \subset S$ one can define an annular arc complex, $\mathcal{C}(\gamma),$ to have vertices corresponding to homotopy classes relative to the boundary of arcs connecting the two boundary components of an annular regular neighborhood of $\gamma,$ and edges between arcs with representatives intersecting once in the annulus.  

Among curves on a surface of finite type we differentiate between two types of curves, namely separating and non-separating curves.  Specifically, a curve $\gamma \subset S$ is \emph{separating} if $S \setminus \gamma$ consists of two connected components, and \emph{non-separating} otherwise.  Given this distinction, we define the \emph{separating curve complex}, denoted $\mathcal{C}_{sep}(S),$ to be the restriction of the usual curve complex to the subset of separating curves.  Notice that for certain low complexity surfaces such as $S_{2,0},$ as defined $\mathcal{C}_{sep}(S)$ is a totally disconnected space, as no two separating curves are disjoint.  Accordingly, in such circumstances, as in the definition of the curve complex for low complexity surfaces, we relax the definition of connectivity and define two separating curves to be connected by an edge if the curves intersect minimally on the given surface.  In particular, for the case of $S_{2,0}$ two separating curves in $\mathcal{C}_{sep}(S_{2,0})$ are connected if and only if they intersect four times. 
\subsection{Essential Subsurfaces, Projections}
An \emph{essential subsurface} $Y$ of a surface $S$ is a subsurface $Y$ with geodesic boundary such that $Y$ is a union of (not necessarily all) complementary components of a multicurve.  An essential subsurface $Y\subseteq S$ is \emph{proper} if $Y\subsetneq S.$  Two essential surfaces $W,V \subset S$ are \emph{disjoint} if they have empty intersection.  On the other hand, we say $W$ is \emph{nested} in $V$, denoted $W \subsetneq V$, if the set of curves that are  contained in $W$ as well as $\partial W$ are contained in $V$.  If W and V are not disjoint, yet neither subsurface is nested in the other, we say that $W$ \emph{overlaps} $V$, denoted $W \pitchfork V $.  

Given a curve $\alpha \in \mathcal{C}(S)$ and a connected essential subsurface $Y\subset S$ with $\xi(Y)\geq 1$ such that $\alpha$ intersects $Y,$  we can define the projection of $\alpha$ to $2^{\mathcal{C}(Y)}$, denoted $\pi_{\mathcal{C}(Y)}(\alpha)$, to be the collection of vertices in $\mathcal{C}(Y)$ obtained by surgering the arcs of $\alpha \cap Y$ along $\partial Y$ to obtain simple closed curves in $Y$.  More formally, the intersection $\alpha \cap Y$ consists of either the curve $\alpha,$ if $\alpha \subset Y,$ or a non-empty disjoint union of arc subsegments of $\alpha$ with the endpoints of the arcs on boundary components of $Y.$  In the former case we define the projection $\pi_{\mathcal{C}(Y)}(\alpha)=\alpha$, whereas in the latter case, $\pi_{\mathcal{C}(Y)}(\alpha)$ consists of all curves in $Y$ obtained by closing up all the arcs in the intersection $\alpha \cap Y$ into curves by taking the union of each of the arcs and a neighborhood of the boundary $\partial Y$, or the \emph{frontier of $Y$}.  See Figure \ref{fig:curveproj} for an example.  Similarly, given any curves $\alpha, \gamma \subset S$ such that $\alpha \cap \gamma \ne \emptyset,$ one can define an annular projection which sends $\alpha$ to $2^{\mathcal{C}(\gamma)}.$  See \cite{MM2} for more details on subsurface and annular projections.  To simplify notation, when measuring distance in the image subsurface complex, we write $d_{\mathcal{C}(Y )} (\alpha_{1}, \alpha_{2})$ as shorthand for 
$d_{\mathcal{C}(Y )} (\pi_{\mathcal{C}(Y)}(\alpha_{1}), \pi_{\mathcal{C}(Y)}(\alpha_{2}))$.

\begin{figure}[htpb]
\centering
\includegraphics[height=2.75 cm]{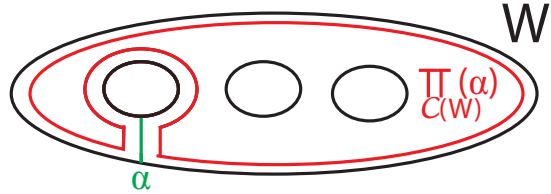}
\caption{Surgering the arc $\alpha \cap W$ along a neighborhood of the boundary of $W$ into a simple closed curve $\pi_{\mathcal{C}(W)}(\alpha) \in \mathcal{C}(W)$  }\label{fig:curveproj}
\end{figure}

\subsection{Combinatorial Complexes and Holes}
In this paper, a \emph{combinatorial complex}, $\mathcal{G}(S),$ will be a graph with vertices defined in terms of multicurves on the surface and edge relations defined in terms of upper bounds on intersections between the multicurves.  In addition, we will assume that combinatorial complexes are invariant under an isometric action of the mapping class group, $\MCG.$  Examples of combinatorial complexes include the separating curve complex, $\mathcal{C}_{sep}(S),$ the arc complex, $\mathcal{A}(S),$ the pants complex, $\mathcal{P}(S),$ the marking complex, $\mathcal{M}(S),$ as well as many others in the literature.  

A \emph{hole} for $\mathcal{G}(S)$ is defined to be any connected essential subsurface such that the entire combinatorial complex has non-trivial subsurface projection into it.  For example, it is not hard to see that holes for the arc complex $\mathcal{A}(S),$ are precisely all connected essential subsurfaces $Y$ such that $\partial S \subset \partial Y.$  

The central idea in \cite{masurschleimer} is that distance in a combinatorial complex is approximated by summing over the distances in the subsurface projections to the curve complexes of holes.  In particular, due to the action by $\MCG,$ if a complex has disjoint holes then the complex admits non-trivial quasi-flats, and hence cannot be $\delta$-hyperbolic.  Conversely, if a combinatorial complex has the property that no two holes are disjoint, then assuming a couple of additional Masur-Schleimer axioms, see \cite{masurschleimer}, the complex is $\delta$-hyperbolic.         

\subsection{Marking Complex and Hierarchy paths} A \emph{complete marking} $\mu$ on $S$ is a collection of \emph{base curves} and \emph{transverse curves} subject to the following conditions: 
\begin{enumerate}
\item The base curves $\{\gamma_{1}, ..., \gamma_{n} \}$ are a maximal dimensional simplex in $\mathcal{C}(S)$.  Equivalently $n=\xi(S)$. 
\item Each base curve $\gamma_{i}$ has a corresponding transversal curve $t_{i},$ transversely intersecting $\gamma_{i},$ such that $t_{i}$ intersects $\gamma_{i}$ exactly once (unless $S=S_{0,4}$ in which case $t_{1}$ intersects $\gamma_{1}$ twice). 
\end{enumerate}  
A complete marking $\mu$ is said to be \emph{clean} if in addition each transverse curve $t_{i}$ is disjoint from all other base curves $\gamma_{j}.$  Let $\mu$ denote a complete clean marking with curve pair data $(\gamma_{i},t_{i}),$ then we define an \emph{elementary move} to be one of the following two operations applied to the marking $\mu:$
\begin{enumerate}
\item \emph{Twist}: For some $i,$ we replace $(\gamma_{i},t_{i})$ with $(\gamma_{i},t'_{i})$ where $t'_{i}$ is the result of one full or half twist (when possible) of $t_{i}$ around $\gamma_{i}.$   
\item \emph{Flip}: For some $i$ we interchange the base and transversal curves.  After a flip move, it is possible that the resulting complete marking is no longer clean, in which case as part of the flip move we then replace the non-clean complete marking with a \emph{compatible} clean complete marking.  Two complete markings $\mu,$ $\mu'$ are compatible if they have the same base curves and moreover $\forall i,$ $d_{\mathcal{C}(\gamma_{i})}(t_{i},t'_{i})$ is minimal over all choices of $t'_{i}.$  In \cite{MM2} it is shown that there is a bound, depending only on the topological type of $S,$ on the number of complete clean markings which are compatible with any given complete marking.  Hence, a flip move is coarsely well defined.   
\end{enumerate}

The \emph{Marking Complex}, $\mathcal{M}(S),$ is defined to be the graph formed by taking complete clean markings of $S$ to be vertices and connecting two vertices by an edge if they differ by an elementary move. 
In \cite{MM2} a 2-transitive family of quasi-geodescis in $\mathcal{M}(S)$, with constants depending only on the topological type of $S$ called \emph{resolutions of hierarchies} are developed.  Informally, hierarchies are defined inductively as of a union of tight geodesics in the curve complexes of essential subsurfaces of $S,$ while resolutions of hierarchies are quasi-geodesics in the marking complex associated to a hierarchy.  By abuse of notation, throughout this paper we will refer to resolutions of hierarchies as hierarchies.  The construction of hierarchies is technical, although for our purposes the following theorem recording some of their properties suffices.
\begin{thm} \label{thm:hierarchy} For $S=S_{g,n},$ $\forall \mu, \nu \in \mathcal{M}(S)$, there exists a hierarchy path $\rho=\rho(\mu,\nu):[0,n] \ra \mathcal{M}(S)$ with $\rho(0)=\mu$, $\rho(n)=\nu$.  Moreover, $\rho$ is a quasi-isometric embedding with uniformly bounded constants depending only on the topological type of $S$, with the following properties: 
\begin{enumerate}
\item  [H1:] The hierarchy $\rho$ \emph{shadows} a tight $\mathcal{C}(S)$ geodesic $g_{S}$ from a multicurve $a \in base(\mu)$ to a multicurve $b\in base(\nu)$, called the \emph{main geodesic of the hierarchy}.  That is, there is a monotonic map $\phi:\rho \ra g_{S}$ such that $\forall i, \; \phi(\rho(i))\in g_{S}$ is a base curve in the marking $\rho(i)$. 
\item [H2:] There is a constant $M_{1}$ such that if an essential subsurface of complexity at least one or an annulus $Y\subset S$ satisfies $d_{\mathcal{C}(Y)}(\mu,\nu) >M_{1}$, then there is a maximal connected interval $I_{Y}=[t_{1},t_{2}]$ and a tight geodesic $g_{Y}$ in $\mathcal{C}(Y)$  from a multicurve in $base(\rho(t_{1}))$ to a multicurve in $base(\rho(t_{2}))$ such that for all $t_{1}\leq t \leq t_{2}$, $\partial Y$ is a multicurve in $base(\rho(t))$ and  $\rho |_{I_{Y}}$ shadows the geodesic $g_{Y}$.  Such a subsurface $Y$ is called a \emph{component domain} of $\rho$.  By convention the entire surface $S$ is always considered a component domain.  
\end{enumerate}
\end{thm}  

The next theorem contains a quasi-distance formula for $\mathcal{M}(S),$ which we later generalize to $\mathcal{C}_{sep}(S_{2,0}).$
\begin{thm} \cite{MM2} \label{thm:quasidistance}For $S=S_{g,n}$ there exists a minimal threshold $C,$ depending only on the surface $S,$ and quasi-isometry constants, depending only on the surface $S$ and the threshold $c\geq C,$ such that $\forall \mu,\nu \mathcal{M}(S)$: 
$$d_{\mathcal{M}(S)}(\mu,\nu) \approx \sum_{Y\subseteq S} \{ d_{\mathcal{C}(Y)}(base(\mu),base(\nu)) \}_{c} $$ 
where the sum is over all essential subsurfaces $Y$ with complexity at least one, as well as all annuli, and the \emph{threshold function} $\{ f(x) \}_{c} :=f(x) \;\; \mbox{ if } f(x) \geq c$ and $0$ otherwise.
\end{thm}

\subsection{Farey Graph}
\begin{figure}[htpb]
\centering
\includegraphics[height=6 cm]{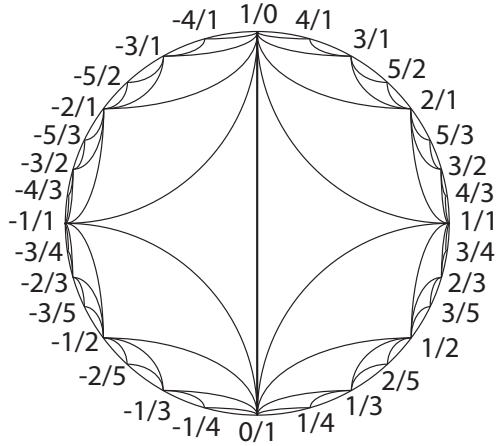}
\caption{A finite portion of the Farey Graph with labeled vertices.}\label{fig:farey}
\end{figure}

The Farey graph is a classical graph with direct application to the study of the curve complex.  Vertices of the Farey graph corresponding to elements of $\Q\cup \{\infty=\frac{1}{0}\},$ with edges between two rational numbers in lowest terms $\frac{p}{q}$ and $\frac{r}{s}$ if $|ps-qr|=1.$  The Farey graph can be drawn as an ideal triangulation of the unit disk as in Figure \ref{fig:farey}.  A nice feature of the Farey graph is the so called \emph{Farey addition property} which ensures that if rational number $\frac{p}{q}$ and $\frac{r}{s}$ are connected in the Farey graph, then there is an ideal triangle in the Farey graph with vertices $\frac{p}{q},$ $\frac{r}{s},$ and $\frac{p+r}{q+s}.$  

The curve complexes $\mathcal{C}(S_{0,4})$ and $\mathcal{C}(S_{1,1})$ are isomorphic to the Farey graph.  The isomorphism is given by sending the positively oriented meridional curve of the surfaces to $\frac{1}{0},$ the positively oriented  longitudinal curve of the surfaces to $\frac{0}{1},$ and more generally sending the $(p,q)$ curve to $\frac{p}{q}.$     

\section{Separating curve complex of the closed genus two surface  is hyperbolic: proof}
\label{sec:proof}
\begin{thm} \label{thm:main}
$\mathcal{C}_{sep}(S_{2,0})$ is $\delta$-hyperbolic
\end{thm}

The proof of Theorem \ref{thm:main} is broken down into two steps.  In the first step we show by a direct argument that $\mathcal{C}_{sep}(S_{2,0})$ has a quasi-distance formula.  In the second step, using step one, we show that the Masur-Schleimer proof for $\delta$-hyperbolicity of a combinatorial complex found in Section 20 of \cite{masurschleimer} applies to $\mathcal{C}_{sep}(S_{2,0})$ despite the fact that not all the Masur-Schleimer axioms hold.  
 
\subsection{Step One: $\mathcal{C}_{sep}(S_{2,0})$ has a quasi-distance formula.}
$\\$
$\indent$
We begin by recalling a lemma of \cite{masurschleimer} which ensures a quasi-lower bound for a quasi-distance formula for $\mathcal{C}_{sep}(S_{2,0}).$  As noted by Masur-Schleimer, the proof of the following lemma follows almost verbatim from the a similar arguments in \cite{MM2} regarding the marking complex:

\begin{lem} \label{lem:lower} Let $S$ be a surface of finite type, and let $\mathcal{G}(S)$ be a combinatorial complex.  There is a constant $C_{0}$ such that $\forall c \geq C_{0}$ there exists quasi-isometry constants such that $\forall \alpha, \beta \in \mathcal{G}(S)$: 
$$\sum_{Y \mbox{ a hole for } \mathcal{G}(S) } \{ d_{\mathcal{C}(Y)}(\alpha,\beta) \}_{c} \lesssim d_{\mathcal{G}(S)}(\alpha,\beta) $$ 
\end{lem} 

In light of Lemma \ref{lem:lower}, in order to obtain a quasi-distance formula for $\mathcal{C}_{sep}(S_{2,0}),$ it suffices to obtain a quasi-upper bound on $\mathcal{C}_{sep}(S_{2,0})$ distance in terms of the sum of subsurface projections to holes.  As motivated by \cite{masurschleimer}, our approach for doing so will be by relating markings to separating curves and more generally marking paths to separating paths.  In the rest of this subsection let $S=S_{2,0}.$

Let $\mu \in \mathcal{M}(S).$  Presently we will define a coarsely well defined mapping $\phi:\mathcal{M}(S) \ra \mathcal{C}_{sep}(S).$  If $base(\mu)$ contains a separating curve $\gamma_{i},$  then we define $\phi(\mu)=\gamma_{i}.$  On the other hand, if all three base curves of $\mu,$ $\gamma_{1},\gamma_{2},\gamma_{3},$ are non-separating curves, then for any $i,j,k \in \{1,2,3\},$ $i\ne j\ne k \ne i,$ denote the essential subsurface $S_{i,j} := S \setminus \gamma_{i}, \gamma_{j} \simeq S_{0,4}.$  Note that $\mathcal{C}(S_{i,j})$ is a Farey graph containing the adjacent curves $\gamma_{k}$ and $t_{k}.$  Let $o_{k}$ be a curve in $S_{i,j}$ such that $\gamma_{k},t_{k}, o_{k}$ form a triangle in $\mathcal{C}(S_{i,j})$.  Note that $o_{k}$ is not uniquely determined by this condition; in fact, there are exactly two possibilities for $o_{k}.$  Nonetheless, the Farey addition property implies that the two possible curves for $o_{k}$ intersect four times and are distance two in $\mathcal{C}(S_{i,j}).$  In this case, assuming none of the base curves are separating curves, we claim that exactly one of $o_{k}$ or $t_{k}$ is a separating curve of $S,$ and define $\phi(\mu)$ to be either $t_{k}$ or $o_{k},$ depending on which one is a separating curve.

\begin{clm} \label{claim:onesep} With the notation from above, let $\gamma_{k},t_{k},o_{k}$ form a triangle in the Farey graph $\mathcal{C}(S_{i,j}).$  Then one (and only one) of the curves $t_{k}$ and $o_{k}$ are separating curves of $S.$ \end{clm}
\begin{proof}
$S_{i,j}$ has four boundary components which glue up in pairs inside the ambient surface $S.$  Moreover, any curve $\alpha \in \mathcal{C}(S_{i,j})$ gives rise to a partition of the four boundary components of $S_{i,j}$ into pairs given by pairing boundary components in the same connected component of $S_{i,j} \setminus \alpha.$  

In total there are ${4 \choose 2} = 3$ different ways to partition the four boundary components of $S_{i,j}$ into pairs, and in fact it is not hard to see that the partition of a boundary components determined by a curve $\frac{p}{q} \in \mathcal{C}(S_{i,j})$ is entirely determined by the parity of $p$ and $q.$  Specifically, the three partitions correspond to the cases (i) $p$ and $q$ are both odd, (ii) $p$ is odd and $q$ is even, and (iii) $p$ is even and $q$ is odd.  By topological considerations, since we are assuming none of the base curves of the marking are separating curves, it follows that all curves  in $\mathcal{C}(S_{i,j})$ corresponding to exactly one of the three cases, (i),(ii) or (iii), are separating curves of the ambient surface $S.$

Hence, in order to prove the claim it suffices to show that any triangle in the Farey graph has exactly one vertex from each of the three cases (i), (ii) and (iii).  This follows from basic arithmetic computation:  First note that no two vertices from a single case are adjacent in the Farey graph.  For example a vertex of type (odd/odd) cannot be adjancent to another vertex of type (odd/odd) as the adjacency condition fails, namely $$|odd^{2}-odd^{2}| = |odd' - odd'| = even \ne 1.$$  Similar calculations show that two vertices of type (odd/even) or two vertices of type (even/odd) cannot be adjacent to each other.  Moreover, the Farey addition property implies that if a triangle contains vertices of two of the different cases, then the third vertex in any such triangle perforce corresponds to the third case.  For example if a triangle has vertices of type (odd/odd) and (odd/even), the Farey addition property implies that the third vertex in any such triangle will be of type (even/odd).  The claim follows.
\end{proof}
  
The following theorem ensures that the mapping $\phi:\mathcal{M}(S) \ra \mathcal{C}_{sep}(S)$ is coarsely well defined.  
\begin{thm}
Using the notation from above, let $\mu$ be a marking with no separating base curves, and let $t_{i},t_{j}$ be transversals which are separating curves.  Then $t_{i}$ and $t_{j}$ are connected in the separating curve complex $\mathcal{C}_{sep}(S).$  Similarly, if $t_{i}$ and $o_{j}$, or $o_{i}$ and $o_{j}$ are separating curves the same result holds.
\end{thm}
\begin{proof}
We will prove the first case; the similar statement follows from the same proof.  Specifically, we will show that the separating curves $t_{i},t_{j}$ intersect four times.  Up to action of $\MCG,$ there is only one picture for a marking $\mu$ which does not contain a separating base curve, as presented in Figure \ref{fig:genus2}.  Without loss of generality we can assume $t_{i}=t_{1}$ and $t_{j}=t_{2}$.  Notice that in the subsurface $S_{2,3},$ as in Figure \ref{fig:genus2}, the base curve $\gamma_{1}$ corresponds to the meridional curve $\frac{1}{0},$ and similarly in the subsurface $S_{1,3}$ the base curve $\gamma_{2}$ also corresponds to the meridional curve $\frac{1}{0}.$  Since $t_{1}$ is connected to $\gamma_{1}$ in the Farey graph $\mathcal{C}(S_{2,3})$ it follows that $t_{1} \in \mathcal{C}(S_{2,3})$ is a curve of the form $\frac{n}{1}$ for some integer $n.$  Similarly,  $t_{2} \in \mathcal{C}(S_{1,3})$ is a curve of the form $\frac{m}{1}$ for some integer $m.$  As in the examples in Figure \ref{fig:genus2} it is easy to draw representatives of the two curves which intersect four times.         
\end{proof}

\begin{figure}[htpb]
\centering
\includegraphics[height=4 cm]{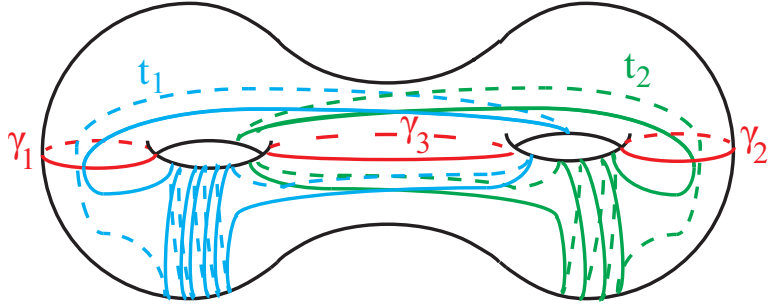}
\caption{A marking $\mu$ on $S=S_{2,0}$ with no separating curves, with transversal curves $t_{1},$ $t_{2}$ separating curves.  Notice that $d_{\mathcal{C}_{sep}(S)}(t_{1},t_{2}) =1.$}\label{fig:genus2}
\end{figure}
   
The following lemma says that our coarsely well defined mapping $\phi$ which associates a separating curve to a complete clean marking is natural with respect to elementary moves in the marking complex.
\begin{lem} \label{lem:natural} If $d_{\mathcal{M}(S)}(\mu,\mu')\leq 1$ then $d_{\mathcal{C}_{sep}(S)}(\phi(\mu),\phi(\mu'))\leq 2.$
\end{lem}
\begin{rem}
To be sure, as will be evident in the proof of the Lemma \ref{lem:natural}, up to choosing appropriate representatives of $\phi(\mu)$ and $\phi(\mu')$ it is in fact true that $d_{\mathcal{C}_{sep}(S)}(\phi(\mu),\phi(\mu'))\leq 1.$  However, the statement of the lemma holds for any representatives of $\phi(\mu)$ and $\phi(\mu').$ 
\end{rem}

\begin{proof}
The proof will proceed by considering cases.  First assume $\mu$ and $\mu'$ differ by a twist move applied to the pair $(\gamma_{i},t_{i})$.  If $\mu$ has a separating base curve, and hence so does $\mu'$ as twists do not affect base curves, then we are done as $\phi$ associates to both markings this common separating base curve.  On the other hand, if $\mu$ has no separating base curves, and hence neither does $\mu',$ we can let $\phi$ assign to both markings the same separating curve either $t_{j}$ or $o_{j},$ for $i\ne j,$ depending on which one is a separating curve.  In either case we are done.

Next assume $\mu$ and $\mu'$ differ by a flip move applied to the pair $(\gamma_{i},t_{i})$.  Recall that after the flip move is performed one must pass to a compatible clean marking.  Let us consider the situation more carefully.  Specifically, assume $\mu=\{(\gamma_{i},t_{i}), (\gamma_{j},t_{j}), (\gamma_{k},t_{k})\}.$  Then $\mu'=\{(t_{i},\gamma_{i}), (\gamma_{j},t'_{j}), (\gamma_{k},t'_{k})\},$ where the transversals $t'_{j},t'_{k}$ are obtained by passing to a compatible clean marking if necessary.  If $\gamma_{j}$ or $\gamma_{k}$ is a separating base curve we are done.  If not, then if $\gamma_{i}$ is a separating curve we are similarly done as $\phi$ can be chosen to assign to both markings the separating curve $\gamma_{i}$.  Finally, if none of the base curves are separating curves, then we also done as we can choose $\phi$ to assign to both markings the same separating curve either $t_{j}$ or $o_{j},$ depending on which one is a separating curve. 
\end{proof}

Combining the existence of well defined mapping $\phi:\mathcal{M}(S) \ra \mathcal{C}_{sep}(S)$ with the result of Lemma \ref{lem:natural}, we have the following procedure for finding a path between any two separating curves.  Given $\alpha,\beta \in \mathcal{C}_{sep}(S),$ complete the separating curves into complete clean markings $\mu$ and $\nu$ such that $\alpha \in base(\mu)$ and $\beta \in base(\nu).$  Then construct a hierarchy path $\rho$ in $\mathcal{M}(S)$ between $\mu$ and $\nu.$  Applying the mapping $\phi$ to our hierarchy path $\rho,$ and interpolating as necessary, yields a path in $\mathcal{C}_{sep}(S)$ between the separating curves $\alpha$ and $\beta$ with length quasi-bounded above by the length of the marking path $\rho.$  In fact, if we are careful we can obtain the following corollary which provides a quasi-upper bound on $\mathcal{C}_{sep}(S_{2,0})$ distance in terms of the sum of subsurface projection to holes.  Note that together with Lemma \ref{lem:lower}, the corollary gives a quasi-distance formula for $\mathcal{C}_{sep}(S),$ thus completing step one.  

\begin{cor} \label{cor:upper}
For $S=S_{2,0},$ there is a constant $K_{0}$ such that $\forall k \geq K_{0}$ there exists quasi-isometry constants such that $\forall \alpha, \beta \in \mathcal{C}_{sep}(S)$: 
$$d_{\mathcal{C}_{sep}(S)}(\alpha,\beta)  \lesssim \sum_{Y \mbox{ a hole for } \mathcal{C}_{sep}(S) } \{ d_{\mathcal{C}(Y)}(\alpha,\beta) \}_{k} $$ 
\end{cor}  

\begin{proof}  As noted, we have a quasi-upper bound on $\mathcal{C}_{sep}(S)$ distance given by the length any hierarchy path $\rho$ connecting markings containing the given separating curves as base curves.  In other words, we have already have a quasi-upper bound of the form: $$d_{\mathcal{C}_{sep}(S)}(\alpha,\beta)  \lesssim \sum_{\xi(Y) \geq 1, \; \mbox{ or } Y \mbox{ an annulus} } \{ d_{\mathcal{C}(Y)}(\alpha,\beta) \}_{k} $$  Hence, it suffices to show that for all components domains $Y$ in the above sum which are not holes of $\mathcal{C}_{sep}(S),$ we can choose can choose our mapping $\phi$ such that the $\mathcal{C}_{sep}(S)$ diameter of $\phi(I_{Y})$ is uniformly bounded, where $I_{Y}$ is as in property [H2] of Theorem \ref{thm:hierarchy}. 

Holes for $\mathcal{C}_{sep}(S)$ consist of all essential subsurfaces with complexity at least one except for subsurfaces whose boundary is a separating curve of the surface.  Hence we must show that for all component domains $Y$ which are either annuli or proper subsurfaces with boundary component a separating curve of the surface that the $\mathcal{C}_{sep}(S)$ diameter of $\phi(I_{Y})$ is uniformly bounded.  First consider the case of $Y$ an annulus.  In this case, the subpath of $\rho$ in the marking complex corresponding to $I_{Y}$ is acting by twist moves on a fixed base curve $\gamma_{i}$.  As in the proof of Lemma \ref{lem:natural}, if there is a separating base curve in the marking, then we are done as the base curves are fixed by the twisting and we can pick the fixed base curve as our separating curve for all of $\phi(I_{Y}).$  Otherwise, if none of the base curves are separating then for $i\ne j$ we pick a $t_{j}$ or $o_{j}$, depending on which is a separating curve, as our representative for all of $\phi(I_{Y}).$  Next consider the case of $Y$ a proper essential subsurface with boundary a separating curve of the surface.  Since every marking in $I_{Y}$ contains the separating curve $\partial Y,$ the desired result follows as we set all of $\phi(I_{Y})$ to be equal to the fixed separating curve $\partial Y.$  
\end{proof}

\subsection{Step Two: $\mathcal{C}_{sep}(S_{2,0})$ is $\delta$-hyperbolic.}
$\\$
$\indent$
In Section 13 of \cite{masurschleimer}, sufficient axioms are established for implying a combinatorial complex admits a quasi-distance formula and furthermore is $\delta$-hyperbolic.  The first axiom is that no two holes for the combinatorial complex are disjoint.  This is easily verified for $\mathcal{C}_{sep}(S_{2,0}).$  The rest of the axioms are related to the existence of an appropriate marking path $\{ \mu_{i} \}_{i=0}^{N} \subset \mathcal{M}(S)$ and a corresponding well suited combinatorial path $\{ \gamma_{i}\}_{i=0}^{K} \subset \mathcal{G}(S).$   In particular, there is a strictly increasing reindexing function $r:[0,K] \ra [0,N]$ with $r(0)=0$ and $r(K)=N.$  In the event that one uses a hierarchy as a marking path, the rest of the axioms can be simplified to the following:
\begin{enumerate}
\item (Combinatorial) There is a constant $C_{2}$ such that for all $i,$ $d_{\mathcal{C}(Y)}(\gamma_{i},\mu_{r(i)})<C_{2}$ for every hole $Y,$ and moreover $d_{\mathcal{G}(S)}(\gamma_{i},\gamma_{i+1})<C_{2}.$
\item (Replacement) There is a constant $C_{4}$ such that:
\subitem[R1] If $Y$ is a hole and $r(i) \in I_{Y},$ then there is a vertex $\gamma' \in \mathcal{G}(S)$ with $\gamma'  \subset Y$ \subitem \;\;\;\; and $d_{\mathcal{G}(S)}(\gamma,\gamma')<C_{4}.$ 
\subitem[R2] If $Y$ is a non-hole and $r(i) \in I_{Y},$ then there is a vertex $\gamma' \in \mathcal{G}(S)$ with $\gamma' \subset Y$ or \subitem \;\;\;\; $\gamma' \subset S\setminus Y$ and $d_{\mathcal{G}(S)}(\gamma,\gamma')<C_{4}.$ 
\item (Straight) There exist constants such that for any subinterval $[p,q] \subset [0,K]$ with the property that $d_{\mathcal{C}(Y)}(\mu_{r(p)},\mu_{r(q)})$ is uniformly bounded for all non-holes, then $d_{\mathcal{G}(S)}(\gamma_{p},\gamma_{q}) \lesssim d_{\mathcal{C}(S)}(\gamma_{p},\gamma_{q}).$
\end{enumerate}
$\indent$ Presently we will show that in the case of the separating curve complex $\mathcal{C}_{sep}(S_{2,0})$ all of above axioms with the exception of axiom [R2] hold.  Let $\rho=\{ \mu_{i} \}_{i=0}^{N}$ be a hierarchy path between two complete clean markings each containing a separating base curve.  Then define the combinatorial path $\{ \gamma_{i}\}_{i=0}^{K} \subset \mathcal{C}_{sep}(S)$ by interpolating between the elements of $\phi(\rho)$ subject to making choices for images of the coarsely well defined mapping $\phi$ such that for component domains of $\rho$ which are not holes of $\mathcal{C}_{sep}(S_{2,0}),$ the $\mathcal{C}_{sep}(S)$ diameter of $\phi(I_{Y})$ is uniformly bounded.  This is precisely what was proven to be possible in Corollary \ref{cor:upper}.  In other words, we can assume the combinatorial path is a quasi-geodesic in the separating curve complex obtained from considering the mapping $\phi$ applied to a hierarchy path $\rho$ and with representative chosen in a manner such that as the hierarchy path potentially travels for an arbitrary distance in a non-hole component domain, the combinatorial path in the separating curve complex only travels a uniformly bounded distance.  Let the reindexing function $r$ be given by sending an element $\gamma_{i}$ of the combinatorial path to any marking $\mu_{j}$ such that $\phi(\mu_{j})=\gamma_{i}.$

Given this setting, the combinatorial axiom is immediate from the definition of $\phi$ in conjunction with Lemma \ref{lem:natural}.  Similarly, the straight axiom follows from the properties of hierarchy paths of Theorem \ref{thm:hierarchy} in conjunction with the construction of the combinatorial path.  Replacement axiom [R1] also holds for if $Y$ is a hole, then $\partial Y$ contains at most two non-separating curves.  Then for all markings $\mu \in I_{Y},$ $base(\mu)$ contains the at most two non-separating curves $\partial Y.$  Let $\gamma_{i}$ be a base curve of $\mu$ not in $\partial Y.$  Then we can choose $\phi(\mu)$ to be either $\gamma_{i},$ $t_{i}$, or $o_{i}$, depending on which is a separating curve, all of which are properly contained in the subsurface $Y.$   Claim \ref{claim:onesep} ensures that exactly one of the three curves $\gamma_{i},$ $t_{i}$, and $o_{i}$ is a separating curve.  On the other hand, axiom [R2] fails as if $Y$ is an essential subsurface which is a non-hole then it is possible that $\partial Y \in \mathcal{C}_{sep}(S).$  In this case, by elementary topological considerations there cannot exist any separating curve properly contained in either $Y$ or $S \setminus Y.$

Nonetheless, while the Masur-Schleimer axioms fail due to the failure of axiom [R2], Masur and Schleimers' proof that a combinatorial complex satisfying the axioms is $\delta$-hyperbolic carries through in the case of $\mathcal{C}_{sep}(S_{2,0}).$  Specifically, consideration of the argument in Section 20 of \cite{masurschleimer}, where Masur and Schleimer prove that a combinatorial complex satisfying their axioms is $\delta$-hyperbolic, reveals the only properties necessary to prove $\delta$-hyperbolcity are that no two holes are disjoint (which holds for $\mathcal{C}_{sep}(S_{2,0}),$ a quasi-distance formula, and the existence of quasi-geodesic combinatorial paths fellow traveling hierarchy paths - in terms of the curve complex of the surface as well the curve complexes of component domain subsurfaces- while effectively avoiding non-hole component domains of the hierarchy path.  However, for the case of $\mathcal{C}_{sep}(S_{2,0}),$ despite the technical failure of the replacement axiom [R2], in step one, and in particular in Corollary \ref{cor:upper}, we have directly proven all of these necessary facts.  Thus, the proof that $\mathcal{C}_{sep}(S_{2,0})$ is $\delta$-hyperbolic follows from the argument in Section 20 of \cite{masurschleimer}, thereby completing the proof of Theorem \ref{thm:main}.   \qed

 \section{Additional Remarks} 
 \subsection{ An alternative proof}
Let $\mathcal{G}(S)$ be a combinatorial complex, and let $\gamma \subset \mathcal{G}(S)$ be a quasi-geodesic.  $\gamma$ has \emph{$(a,b,c)$-contraction}, if there exists positive constants $a,b,c$ such that $d_{\mathcal{G}(S)}(\alpha,\gamma) \geq a, \; d_{\mathcal{G}(S)}(\alpha,\beta) \leq b d_{\mathcal{G}(S)}(\alpha,\gamma),$ implies $d_{\mathcal{G}(S)}(\pi_{\gamma}(\alpha),\pi_{\gamma}(\beta)) \leq c,$ where $\alpha, \beta \in \mathcal{G}(S)$ and $\pi_{\gamma}: \mathcal{P}(S) \ra 2^{\gamma}$ is a nearest point projection.  In \cite{MM1}, as well as independently in \cite{behrstock}, it is proven that a metric space with a transitive family of quasi-geodesics with $(a,b,c)$-contraction is $\delta$-hyperbolic.  

In \cite{behrstock}, it is shown that the complexes $\mathcal{M}(S_{0,4}), \mathcal{M}(S_{1,1}), \mathcal{P}(S_{0,5}), \mathcal{P}(S_{1,2})$ all have transitive families of quasi-geodesics with $(a,b,c)$-contraction, which in particular implies they are $\delta$-hyperbolic.  In the opinion of the author, the methods of Behrstock in Section 5 of \cite{behrstock} applied to the above four complexes with appropriate modification could be applied to give the same conclusion regarding $\mathcal{C}_{sep}(S_{2,0}).$  This would provide an alternative proof of the fact that $\mathcal{C}_{sep}(S_{2,0})$ is $\delta$-hyperbolic.  We refer the reader to that paper, and presently limit ourselves to pointing out a couple of technical issues which must be addressed in the course of adapting the Behrstock argument in Section 5 of \cite{behrstock} to $\mathcal{C}_{sep}(S_{2,0}).$

Specifically, in Lemma 5.3 of \cite{behrstock} instead of $\mathcal{G} \setminus \{\mathcal{L} \cup \mathcal{R}\} $ consisting of at most one component domain, in the case of $\mathcal{C}_{sep}(S),$ the set $\mathcal{G} \setminus \{\mathcal{L} \cup \mathcal{R}\} $ now consists of at most two component domains.  In the event of two component domains, the domains are nested.  Moreover, in light of the previous remark, the coarsely well defined projection $\hat{\Phi}$ must include an additional case corresponding to where $\mathcal{G} \setminus \{\mathcal{L} \cup \mathcal{R}\} $ consists of two nested subsurfaces.  More generally, the definition of the projection $\hat{\Phi}$ must be modified to to have output a separating curve, as opposed to a marking, while making sure the desired properties of the map are unaffected.

\subsection{A quasi-distance formula for $\mathcal{C}_{sep}(S)$ in general}  Considering the arguments in step one of Section \ref{sec:proof}, one is tempted to believe that they can be appropriately modified to provide a proof of a quasi-distance formula for  $\mathcal{C}_{sep}(S)$ in general.  However, this is certainly not immediate.  Specifically, an explicit construction in Section 5 of \cite{sultan} implies that, for high enough genus, there exist complete clean markings of closed surfaces which are arbitrarily far (with respect to elementary moves) from any complete clean marking containing a separating base or transversal curve.  This is in stark contrast with the situation in $\mathcal{C}_{sep}(S_{2,0}),$ for which in Section \ref{sec:proof}, we make strong use of the fact that any complete clean marking is distance at most one from a complete clean marking containing a separating base or transversal curve.


\end{document}